\documentclass[12pt,oneside]{amsart}

\usepackage{amssymb}
\usepackage{amscd}

\title[Plane curves with two or more Galois points]{A family of plane curves with two or more Galois points in positive characteristic}
\author{Satoru Fukasawa}

\address{Department of Mathematical Sciences,  
Faculty of Science, Yamagata University, 
Kojirakawa-machi 1-4-12, Yamagata 990-8560, Japan.}
\email{s.fukasawa@sci.kj.yamagata-u.ac.jp} 

\newtheorem{theorem}{Theorem}
\newtheorem{proposition}{Proposition}

\newtheorem{fact}{Fact}

\newtheorem{lemma}{Lemma} 

\newcommand{\fq}{\mathbb{F}_q}

\theoremstyle{definition}
\newtheorem{example}{Example}
\newtheorem{remark}{Remark}

\begin{document}
\begin{abstract} 
We give new examples of plane curves with two or more Galois points as a family, and describe the number of Galois points for these curves, by using finite fields. 
\end{abstract}
\maketitle
\section{Introduction}  
In this paper, we give new examples of plane curves with two or more Galois points as a family, and describe the number of Galois points for these curves, by using finite fields. 

Let $K$ be an algebraically closed field of characteristic $p \ge 0$ and let $C \subset \mathbb P_K^{2}$ be an irreducible plane curve of degree $d \ge 3$. 
Hisao Yoshihara introduced the notion of {\it Galois point} in 1996 (see \cite{miura-yoshihara, yoshihara}). 
If the function field extension $K(C)/K(\mathbb P^1)$, induced by the projection $\pi_P:C \dashrightarrow \mathbb P^1$ from a point $P \in \mathbb P^{2}$, is Galois, then the point $P$ is said to be Galois with respect to $C$. 
When a Galois point $P$ is contained in $\mathbb P^{2} \setminus C$, we call $P$ an outer Galois point. 
We denote by $\delta'(C)$ the number of outer Galois points for $C$. 
It would be interesting to determine $\delta'(C)$. 
For example, there are applications of the distribution of Galois points to finite geometry (see \cite{fhk}). 
When $C$ is smooth, $\delta'(C)$ is completely determined (see \cite{fukasawa3, homma, miura-yoshihara, yoshihara}). 
However, it is difficult to determine $\delta'(C)$ for (singular) curves $C$ in general. 
For example, curves $C$ satisfying $\delta'(C) >1$ are very rare \cite{open}. 

In this paper, we give new examples. 
Let $p>0$, $e \ge 1$ and let 
\begin{eqnarray*}
g_1(x)&:=&x^{p^e}+\alpha_{e-1}x^{p^{e-1}}+\cdots+\alpha_1x^{p}+\alpha_0x, \\
& & g_2(y):=y^{p^e}+\beta_{e-1}y^{p^{e-1}}+\cdots+\beta_1y^{p}+\beta_0y, 
\end{eqnarray*}
where $\alpha_{e-1}, \ldots, \alpha_0, \beta_{e-1}, \ldots, \beta_0 \in K$ and $\alpha_0\beta_0 \ne 0$. 
Assume that $\ell \ge 2$ is an integer such that $\ell$ is not divisible by $p$, $\ell$ divides $p^e-1$, and $\ell$ divides $p^i-1$ if $\alpha_i \ne 0$ or $\beta_i \ne 0$. 
We consider the curve $C \subset \mathbb{P}^2$ of degree $p^e \ell$ (which is the projective closure of the affine curve) defined by 
\begin{equation} \label{twoGalois} 
 g_1(x)^\ell+\lambda g_2(y)^\ell+\mu=0, \tag{I}
\end{equation}  
where $\lambda, \mu \in K\setminus \{0\}$. 
The main theorem describes the number of Galois points in terms of finite fields, as follows. 

\begin{theorem} \label{main} 
Let the characteristic $p>0$, $q=p^e$, $C \subset \mathbb{P}^2$ be the plane curve given by equation $(\ref{twoGalois})$ and let $\alpha \in K$ satisfy $\alpha^{q\ell}+\lambda=0$. 
Then, we have the following. 
\begin{itemize}
\item[(a)] The curve $C$ is irreducible and has exactly $\ell$ singular points. 
\item[(b)] The genus of the smooth model satisfies $p_g(C) \ge \frac{q\ell(q(\ell-1)-2)}{2}+1$. 
Furthermore, the equality holds if and only if $g_1(\alpha X)-\alpha^{q} g_2(X)=0$ as a polynomial. 
\item[(c)] $\delta'(C) \ge 2$. 
\item[(d)] If $\ell$ is odd, $\ell-1$ is not divisible by $p$ and $g_1(\alpha X)-\alpha^{q} g_2(X)=0$, then $\delta'(C)=2$. 
\item[(e)] Let $\ell=2$, $g_1(X)=g_2(X)$, $\mathbb F_{q_0}:=\mathbb F_{q} \cap (\bigcap_{\{i>0 :\alpha_i \ne 0\}} \mathbb F_{p^i})$ and let $\lambda \in \mathbb F_{q_0}$. 
If $\alpha \in \mathbb F_{q_0}$ (resp. $\alpha \not\in \mathbb F_{q_0}$), then $\delta'(C)\ge q_0-1$ (resp. $\delta'(C) \ge q_0+1$). 
Furthermore, if $q_0=q$, then the equality holds. 
\end{itemize} 
\end{theorem}

The proof of assertion (e) is done by using the projective geometry over a finite field $\fq$.  
It would be interesting that the set of Galois points is described as the set of $\fq$-rational points $\not\in C$ on the line $Z=0$ (see Lemmas \ref{rationality0}, \ref{rationality1} and \ref{rationality2}). 

As another aspect, our curves have non-trivial automorphism groups. 
For example, a famous curve studied by Subrao belongs to our family.  

\begin{example} 
Taking $p >2$, $\ell=2$, $g_1(x)=x^q-x$, $g_2(y)=y^q-y$ and $\lambda=-1$ and making the variable change $X=x+y$ and $Y=x-y$ we obtain the equation 
$$ 
(X^q-X)(Y^q-Y)+\mu=0 
$$  
studied by D. Subrao \cite{subrao}, \cite[Example 11.89]{hkt}, and this is known as an example of an irreducible (ordinary) curve whose automorphism group exceeds the Hurwitz bound under certain assumptions. 
For Galois points in $p=2$, see \cite{fukasawa4}. 
\end{example} 

\begin{remark}
The curve with equation (\ref{twoGalois}) and $\ell=1$ was studied in \cite{fukasawa2}. 
\end{remark}

\section{Preliminaries} 
We introduce a system $(X:Y:Z)$ of homogeneous coordinates on $\Bbb P^2$ with local coordinates $x=X/Z$, $y=Y/Z$ for the affine piece $Z \ne 0$. 
For an irreducible plane curve $C \subset \mathbb{P}^2$, we denote by ${\rm Sing}(C)$ the singular locus of $C$, by $\pi: \hat{C} \rightarrow C$ the normalization of $C$ and by $\hat{\pi}_P$ the composition map $\pi_P \circ \pi: \hat{C} \rightarrow \mathbb P^1$. 
For a point $R \in C \setminus {\rm Sing}(C)$, $T_RC \subset \mathbb P^2$ is the (projective) tangent line at $R$. 
For a projective line $L \subset \mathbb P^2$ and a point $R \in C \cap L$, $I_R(C, L)$ means the intersection multiplicity of $C$ and $L$ at $R$.  
We denote by $\overline{PR}$ the line passing through points $P$ and $R$ when $P \ne R$. 
If $\hat{R} \in \hat{C}$, we denote by $e_{\hat{R}}$ the ramification index of $\hat{\pi}_P$ at $\hat{R}$. 
If $R \in C \setminus {\rm Sing}(C)$ and $\pi(\hat{R})=R$, then we use the same symbol $e_R$ for $e_{\hat{R}}$, by abuse of terminology.  
We note the following elementary fact.  
 
\begin{fact} \label{index}
Let $P \in \mathbb P^2 \setminus C$, let $\hat{R} \in \hat{C}$ and let $\pi(\hat{R})=R$. 
Let $h=0$ be a local equation for the line $\overline{PR}$ in a neighborhood of $R$. 
Then, for $\pi_P$ we have 
$$e_{\hat{R}}={\rm ord}_{\hat{R}}(\pi^*h). $$  
In particular, if $R$ is smooth, then $e_R=I_R(C, \overline{PR})$. 
\end{fact}

The following fact is useful (see \cite[III. 7.2]{stichtenoth}). 
\begin{fact} \label{Galois covering} 
Let $C, C'$ be smooth curves, let $\theta: C \rightarrow C'$ be a Galois covering of degree $d$ and let $R, R' \in C$. 
Then we have the following. 
\begin{itemize}
\item[(a)] If $\theta(R)=\theta(R')$, then $e_R=e_{R'}$. 
\item[(b)] The index $e_R$ divides the degree $d$. 
\end{itemize}
\end{fact} 

Note that the polynomial $g(X)=X^{p^e}+\alpha_{e-1}X^{p^{e-1}}+\cdots+\alpha_1X^{p}+\alpha_0X \in K[X]$
has the property 
$$ g(X+Y)=g(X)+g(Y) \ \ \ (\in K[X, Y]). $$

\section{Proof of assertions $(a)$, $(b)$ and $(c)$ in Theorem \ref{main}} 
Let $p>0$, let $\ell \ge 2$ be an integer not divisible by $p$ and let $g(x)=x^{p^e}+\alpha_{e-1}x^{p^{e-1}}+\cdots+\alpha_1x^p+\alpha_0x$, where $\alpha_{e-1}, \ldots, \alpha_0 \in K$ and $\alpha_0 \ne 0$. 
We consider the plane curve $C \subset \mathbb P^2$ (which is the projective closure of the affine curve) defined by 
$$ f(x, y):=g(x)^{\ell}+h(y)=0, $$
where $h(y)$ is a polynomial of degree $p^e\ell$. 
Assume that $\ell$ divides $p^e-1$, and $p^i-1$ if $\alpha_i \ne 0$. 
For a constant $a \in K$ with $g(a)=0$, we denote by $\sigma_a$ the linear transformation defined by $(X:Y:Z) \mapsto (X+aZ:Y:Z)$. 
Let $q=p^e$, $P=(1:0:0)$ and let $K_P:=\{\sigma_a: g(a)=0 \}$. 
Then, $K_P$ is a subgroup of ${\rm Aut}(\mathbb P^2)$ of order $q$. 
We take a primitive $\ell$-th root $\zeta$ of unity. 
Let $\tau$ be the linear transformation given by $(X:Y:Z) \mapsto (\zeta X:Y:Z)$ and let $G:=K_P\langle\tau\rangle$. 
Since $g(\zeta a)=\zeta g(a)=0$ if $g(a)=0$, we have $\langle\tau\rangle K_P=K_P\langle \tau \rangle$. 
This implies that $G$ is a subgroup of ${\rm Aut}(\mathbb P^2)$ of order $q \ell$. 
Let $L$ be a line passing through $P$. 
Then, $L$ is given by $cY+dZ=0$ for some $c, d \in K$. 
Since $\sigma_a(x:-d:c)=(x+ac:-d:c)$ and $\tau(x:-d:c)=(\zeta x:-d:c)$, we have $\xi (L)=L$ for any $\xi \in G$. 
Since $g(x+a)=g(x)+g(a)=g(x)$ for any $a$ with $g(a)=0$ and $g(\zeta x)=\zeta g(x)$, $\xi(C)=C$ for any $\xi \in G$. 
The group $G$ will be the Galois group $G_P$ at $P$ if one is able to prove that $C$ is irreducible, because $G$ gives $q\ell$ automorphisms of $C$ which commute with the projection $\pi_P: C \rightarrow \mathbb P^1$. 

\begin{proposition} \label{irreducibility} 
Let $C \subset \mathbb{P}^2$ be given by $f(x, y)=g(x)^\ell+h(y)=0$. 
If $h(y)$ is a separable polynomial of degree $q \ell$, then $C$ is irreducible.
Furthermore, the genus of the smooth model satisfies $p_g(C) \ge \frac{q\ell(q(\ell-1)-2)}{2}+1$. 
\end{proposition}

\begin{proof}
Let $C_0$ be an irreducible component of $C$ with degree $d_0$ and the genus of the smooth model $p_g(C_0)$. 
Since $f_x=\ell g(x)^{\ell-1}$ and $f_y=h_y$, $C$ is smooth in the affine plane $Z \ne 0$. 
In this affine plane, $C_0$ is smooth and $C_0$ does not intersect another irreducible component. 
Let $F(X, Y, Z)=Z^{q\ell}f(X/Z, Y/Z)$, $G(X, Z)=Z^{q}g(X/Z)$ and let $H(Y, Z)=Z^{q\ell}h(Y/Z)$. 
For $\xi \in K$ with $h(\xi)=0$, we denote by $L_{\xi}$ the line defined by $Y-\xi Z=0$. 
The scheme $C \cap L_{\xi}$ is given by $F(X, \xi Z, Z)=G(X, Z)^\ell=0$. 
Since $G(X, Z)$ is a separable polynomial of degree $q$, the multiplicity of the set $C \cap L_{\xi}$ is $\ell$ at each point and the set $C \cap L_{\xi}$ consists of exactly $q$ points. 
Therefore, the multiplicity of the set $C_0 \cap L_{\xi}$ is $\ell$ at each point and the set $C_0 \cap L_{\xi}$ consists of exactly $d_0/\ell$ points.
We consider the projection $\pi_P:C_0 \rightarrow \mathbb P^1$ from $P$. 
Since each point of $C_0 \cap L_{\xi}$ is tame ramification with index $\ell$, we have
$$ 2p_g(C_0)-2 \ge d_0(-2)+q\ell \times (d_0/\ell)(\ell-1)$$
by the Riemann-Hurwitz formula. 
Then, $2p_g(C_0)-2 \ge d_0(q(\ell-1)-2)$. 
On the other hand, by genus formula, we have $2p_g(C_0)-2 \le (d_0-3)d_0$. 
Therefore,  $d_0 \ge q(\ell-1)+1$.  
We find that any irreducible component of $C$ is of degree at least $(q\ell)/2+1$. 
Since the degree of $C$ is $q\ell$, $C$ is irreducible. 

Using the Riemann-Hurwitz formula again, we have the inequality $p_g(C) \ge \frac{q\ell(q(\ell-1)-2)}{2}+1$. 
\end{proof}

We consider the plane curves with equation (\ref{twoGalois}) in Introduction. 
Points $(1:0:0)$, $(0:1:0)$ are outer Galois, by the argument before Proposition \ref{irreducibility}. 
We have $\delta'(C) \ge 2$.  

\begin{remark}
The Galois group $G$ at $(1:0:0)$ or $(0:1:0)$ is isomorphic to $(\mathbb Z/p\mathbb Z)^{\oplus e} \rtimes \langle \zeta \rangle$, since we have the splitting exact sequence of groups (see \cite[Theorem 2]{fukasawa1}): 
$$ 0 \rightarrow (\mathbb Z/p\mathbb Z)^{\oplus e} \rightarrow G \rightarrow \langle \zeta \rangle \rightarrow 1. $$
\end{remark}

\begin{proposition} \label{general property}
Let $C \subset \mathbb{P}^2$ be given by $f(x, y):=g_1(x)^{\ell}+\lambda g_2(y)^{\ell}+\mu=0$ and let $\alpha \in K$ satisfy $\alpha^{q\ell}+\lambda=0$. 
Then, we have the following. 
\begin{itemize}
\item[(a)] We have ${\rm Sing}(C)=\{(\zeta^i \alpha:1:0): 0 \le i \le \ell-1\}$. In particular, $C$ has exactly $\ell$ singular points. 
\item[(b)] Assume that $g_1(\alpha X)-\alpha^{q} g_2(X)=0$ as a polynomial. 
Then, $\pi^{-1}(Q)$ consists of exactly $q$ points and there exist $q$ distinct tangent directions at $Q$, for any $Q \in {\rm Sing}(C)$.
In this case, $p_g(C)=\frac{q\ell(q(\ell-1)-2)}{2}+1$.   
\item[(c)] Assume that $g_1(\alpha X)-\alpha^q g_2(X) \ne 0$. 
Then, the multiplicity at $Q$ is at most $q-1$ and $\pi^{-1}(Q)$ consists of at most $q/v$ points for some $v$ a power of $p$. 
In this case, the genus of the smooth model satisfies $p_g(C) \ge \frac{q\ell(q(\ell-1)-1)}{2}+1$. 
\end{itemize}
\end{proposition}

\begin{proof}
We prove (a). 
Let $F(X, Y, Z)=Z^{q\ell}f(X/Z, Y/Z)$, $G_1(X, Z)=Z^{q}g_1(X/Z)$ and $G_2(Y, Z)=Z^{q}g_2(Y/Z)$. 
We have 
\begin{eqnarray*} 
& & \frac{\partial F}{\partial X}=\ell G_1(X, Z)^{\ell-1}\alpha_0Z^{q-1},\ \frac{\partial F}{\partial Y}=\lambda\ell G_2(Y, Z)^{\ell-1}\beta_0Z^{q-1}, \\
& & \frac{\partial F}{\partial Z} 
=\ell(q-1) G_1(X, Z)^{\ell-1}\alpha_0XZ^{q-2}+\lambda\ell(q-1)G_2(Y, Z)^{\ell-1}\beta_0YZ^{q-2}. 
\end{eqnarray*} 
Therefore, the singular locus is given by $Z=X^{q\ell}+\lambda Y^{q\ell}=0$. 
We have ${\rm Sing}(C)=\{(\zeta^i \alpha:1:0): 0 \le i \le \ell-1\}$, which consists of exactly $\ell$ points. 

We prove (b). 
Note that ${\rm Sing}(C)=\{(\zeta^i \alpha:1:0): 0 \le i \le \ell-1\}$ by (a).
Let $Q=(\alpha:1:0)$. 
We consider the projection $\pi_Q$ from $Q$. 
If we take $t:=x-\alpha y$, then the projection $\pi_Q$ is given by $\pi_R(x:y:1)=(t:1)$ and we have a field extension $K(t, y)/K(t)$ with the equation $\hat{f}(t, y):=g_1(t+\alpha y)^\ell+\lambda g_2(y)^\ell+\mu=0$. 
Note that 
$$g_1(t+\alpha y)^\ell=(g_1(t)+g_1(\alpha y))^{\ell}=\sum_{i=0}^{\ell}\binom{\ell}{i} g_1(t)^{\ell-i} g_1(\alpha y)^i, $$
and 
$$\hat{f}(t, y)=\sum_{i=0}^{\ell-1}\binom{\ell}{i} g_1(t)^{\ell-i} g_1(\alpha y)^i+g_1(\alpha y)^\ell+\lambda g_2(y)^\ell+\mu. $$
Note that
$$ g_1(\alpha y)^{\ell}+\lambda g_2(y)^\ell=\prod_{0 \le i \le \ell-1}(g_1(\alpha y)-\zeta^i\alpha^{q} g_2(y)). $$ 
Since $g_1(\alpha y)-\zeta^i\alpha^qg_2(y)=(\alpha^{q}-\zeta^i \alpha^q)y^{q}+(\mbox{ lower terms })$
and $\alpha^{q}-\zeta^i \alpha^q \ne 0$ if $0<i \le \ell-1$, 
the degree of $\hat{f}(t,y)$ as a polynomial over $K(t)$ is $> q(\ell-1)$ if and only if $g_1(\alpha y)-\alpha^{q} g_2(y) \ne 0$ as a polynomial. 

Assume that $g_1(\alpha X)-\alpha^{q} g_2(X)=0$. 
Then, we have $g_1(\alpha X)^{\ell}+\lambda g_2(X)^\ell=0$ and the equation  
$$\sum_{i=0}^{\ell-1}\binom{\ell}{i} g_1(t)^{\ell-1-i} g_1(\alpha y)^i+\frac{\mu}{g_1(t)}=0.  $$
According to \cite[III. 1. 14]{stichtenoth}, for each solution $t_0$ of $g_1(t)=0$, $\hat{\pi}_Q^{-1}((t_0:1))$ consists of a single point $\hat{Q}_{t_0} \in \hat{C}$ with $e_{\hat{Q}_{t_0}}=q(\ell-1)$ for the projection $\hat{\pi}_{Q}$. 
Therefore, $\pi^{-1}(Q)$ consists of exactly $q$ points and there exist exactly $q$ distinct tangent directions. 
We consider the projection $\pi_P$ from $P=(1:0:0)$. 
The set $\hat{\pi}_P^{-1}(\pi_P(Q))=\pi^{-1}(C \cap \{Z=0\})$ consists of exactly $q \ell$ points. 
Therefore, $\hat{\pi}_P$ is not ramified at any point in $\pi^{-1}(C \cap \{Z=0\})$. 
By the Riemann-Hurwitz formula, we have $p_g(C)=\frac{q\ell(q(\ell-1)-2)}{2}+1$. 

Assume that $g_1(\alpha X)-\alpha^{q} g_2(X) \ne 0$.  
Then, the degree of the projection $\pi_Q$ is at least $q(\ell-1)+1$. 
This implies the multiplicity at $Q$ is at most $q-1$ and the cardinality of $\pi^{-1}(C \cap \{Z=0\})$ is strictly less than $q\ell$. 
When we consider the projection $\pi_P$ from $P=(1:0:0)$, there exists a ramification point in $\pi^{-1}(C \cap \{Z=0\})$. 
Let $v$ be the index at the point. 
By Fact \ref{Galois covering}(a)(b), $v$ is the index for any point in $\pi^{-1}(C \cap \{Z=0\})$ and $v$ divides $p^e\ell$. 
Since $C \cap \{Z=0\}$ consists of exactly $\ell$ points, $v$ divides $q$, and $\pi^{-1}(Q)$ consists of at $q/v$ points. 
By the Riemann-Hurwitz formula, $$2p_g(C)-2 \ge q\ell(-2)+q\ell\times\frac{q\ell}{\ell}(\ell-1)+\frac{q\ell}{v}\times v.$$ 
Then, $p_g(C) \ge \frac{q\ell(q(\ell-1)-1)}{2}+1$. 
\end{proof} 

\begin{example} 
There exist curves whose genus attains the lower bound as in assertion (b) for each $q, \ell$. 
Let $g_1(X)=g_2(X)$ and $\lambda=-1$. 
Then, $\alpha =1 \in K$ satisfies $\alpha^{q\ell}+\lambda=1-1=0$ and $g_1(\alpha X)-\alpha^{q} g_2(X)=g_1(X)-g_1(X)=0$. 
\end{example}

\section{Proof of assertion $(d)$ in Theorem \ref{main}}

Assume that $\ell$ is odd, $\ell-1$ is not divisible by $p$ and $g_1(\alpha X)-\alpha^{q} g_2(X)=0$. 
Then, $p>2$ and $p_g(C)=\frac{q\ell(q(\ell-1)-2)}{2}+1$. 
Moreover, $m \ge 3$ if $m$ divides $q\ell$.   
We prove $\delta'(C)=2$. 
For a smooth point $R$, we call $R$ a flex if $I_R(C, T_RC) \ge 3$. 

\begin{lemma} \label{flex} 
Let $R=(x_0:y_0:1) \in C \setminus {\rm Sing}(C)$. 
Then, $R$ is a flex if and only if $g_1(x_0)=0$ or $g_2(y_0)=0$. 
In this case, $I_R(C, T_RC)=\ell$ and $T_RC$ passes through $(1:0:0)$ or $(0:1:0)$. 
\end{lemma} 

\begin{proof}
We consider the Hessian matrix 
\begin{eqnarray*} 
H(f)&=&\left(\begin{array}{ccc}
f_{xx} & f_{xy} & f_{x} \\
f_{xy} & f_{yy} & f_{y} \\
f_{x} & f_{y} & 0 
\end{array}\right) \\
&=&
\left(\begin{array}{ccc} 
\ell(\ell-1)g_1(x)^{\ell-2}\alpha_0^2 & 0 & \ell g_1(x)^{\ell-1}\alpha_0 \\
0 & \lambda \ell(\ell-1)g_2(y)^{\ell-2}\beta_0^2 & \lambda\ell g_2(y)^{\ell-1}\beta_0 \\
\ell g_1(x)^{\ell-1}\alpha_0 & \lambda \ell g_2(y)^{\ell-1}\beta_0 & 0 
\end{array} \right). 
\end{eqnarray*}
Then, $\det H(f)=-\lambda \alpha_0^2\beta_0^2\ell^3(\ell-1)g_1(x)^{\ell-2}g_2(y)^{\ell-2}(\lambda g_2(y)^\ell+g_1(x)^\ell)$. 
It is well known that $R \in C \setminus {\rm Sing}(C)$ is a flex if and only if $\det H(f)(R)=0$ (see, for example, \cite[I.1.5]{shafarevich}).  
Since $\ell(\ell-1)$ is not divisible by $p$ by the assumptions, $R$ is a flex if and only if $g_1(x_0)=0$ or $g_2(y_0)=0$. 

If $g_1(x_0)=0$, then $g_2(y_0)^\ell+\mu=0$, by the defining equation. 
Since $f(x,y_0)=g_1(x)^\ell$, $T_RC$ is defined by $Y-y_0Z$ and $I_R(C, T_RC)=\ell$. 
This line passes through $(1:0:0)$. 
\end{proof}

\begin{proof}[Proof of Theorem \ref{main}(d)]
Assume by contradiction that a point $P'$ in the affine plane $Z \ne 0$ is outer Galois. 
Let $Q$ be a singular point. 
Since $\pi^{-1}(Q)$ consists of exactly $q$ points by Proposition \ref{general property}(b), $\hat{\pi}_{P'}$ is ramified at some point in $\pi^{-1}(Q)$ if and only if $P'$ lies on the line given by some tangent direction at $Q$. 
In this case, $P'$ is ramified at one point in $\pi^{-1}(Q)$ and unramified at the other $q-1$ points. 
By Fact \ref{Galois covering}(a), this is a contradiction. 
Therefore, we may assume that $\hat{\pi}_{P'}$ is unramified at each point in $\pi^{-1}(Q)$ for any singular point $Q$. 
Note that $\hat{\pi}_{P'}$ is ramified at some point $\hat{R} \in \hat{C}$, by the Riemann-Hurwitz formula.  
Now, $R=\pi(\hat{R}) \in C \setminus {\rm Sing}(C)$. 
By Fact \ref{Galois covering}(b) and the assumption on degree, $I_R(C, T_RC) \ge 3$. 
According to Lemma \ref{flex}, $I_R(C, T_RC)=\ell$. 
By the Riemann-Hurwitz formula again, there exist at least three points $R_1, R_2, R_3$ such that $T_{R_i}C \ni P'$ for each $i$ and $T_{R_i}C \ne T_{R_j}C$ if $i \ne j$. 
By Lemma \ref{flex}, $P'$ must be $(1:0:0)$ or $(0:1:0)$. 
This is a contradiction. 

Assume that $P' \in \{Z=0\}$ is outer Galois. 
Since $\pi^{-1}(C \cap \{Z=0\})$ consists of exactly $q\ell$ points by Proposition \ref{general property}(b), the projection $\hat{\pi}_{P'}$ is unramified at such points. 
Note that $\hat{\pi}_{P'}$ is ramified at some point $\hat{R} \in \hat{C}$, by the Riemann-Hurwitz formula.  
Now, $R=\pi(\hat{R}) \in C \setminus {\rm Sing}(C)$. 
By Fact \ref{Galois covering}(b) and the assumption on degree, $I_R(C, T_RC) \ge 3$. 
By Lemma \ref{flex}, $P'$ must be $(1:0:0)$ or $(0:1:0)$. 
\end{proof} 

\section{Proof of assertion $(e)$ in Theorem \ref{main}}

Let $\ell=2$, let $g_1(X)=g_2(X)=g(X):=X^{p^e}+\alpha_{e-1}X^{p^{e-1}}+\dots+\alpha_0X$ and let $\mathbb F_{q_0}:=\fq \cap (\bigcap_{\{i>0:\alpha_i \ne 0\}}\mathbb F_{p^i})$. 
Then, $p>2$, since $p$ does not divide $\ell=2$. 
We consider the curve $C$ defined by 
$$ g(x)^2+\lambda g(y)^2+\mu=0, $$
where $\lambda \in \mathbb F_{q_0}$.
Note that $g(\gamma x)=\gamma g(x)$ if $\gamma \in \mathbb F_{q_0}$, and $\alpha^2=-\lambda$, since $\alpha^{2q}=-\lambda=-\lambda^q$.  
The following Lemma implies the former assertion of (e). 

\begin{lemma} \label{rationality0}
Let $\gamma \in \mathbb F_{q_0}$ and let $P'=(\gamma:1:0)$.
If $\gamma^{2}+\lambda \ne 0$, then $P'$ is outer Galois. 
In particular, $\delta'(C) \ge q_0-1$, and $\delta'(C) \ge q_0+1$ if $\alpha \not\in \mathbb F_{q_0}$. 
\end{lemma}

\begin{proof}
The projection $\pi_{P'}$ from $P'$ is given by $\pi_{P'}(x:y:1)=(x-\gamma y:1)$. 
Let $t:=x-\gamma y$. 
Then, we have the field extension $K(t, y)/K(t)$ given by $\hat{f}(t, y):=g(t+\gamma y)^2+\lambda g(y)^2+\mu=0$. 
Since $g(\gamma x)=\gamma g(x)$ and $\gamma^2+\lambda \ne 0$, we have
\begin{eqnarray*} 
\hat{f}(t, y) &=& (g(t)+\gamma g(y))^2+\lambda g(y)^2+\mu \\
&=&(\gamma^{2}+\lambda)g(y)^2+2\gamma g(t)g(y)+g(t)^2+\mu \\
&=&(\gamma^{2}+\lambda)g(y)\left(g(y)+\frac{2\gamma}{\gamma^2+\lambda}g(t)\right)+g(t)^2+\mu.
\end{eqnarray*} 
Let $\eta:=\frac{2\gamma}{\gamma^2+\lambda} \in \mathbb F_{q_0}$. 
Note that $\eta^{q_0}=\eta$ and $g(\eta \times x)=\eta \times g(x)$. 
Then, the set $\{y+\beta: g(\beta)=0\}\cup\{-y-\eta t+\beta:g(\beta)=0\} \subset K(t, y)$ consists of all roots of $\hat{f}(t, y) \in K(t)[y]$. 
Therefore, $P'$ is outer Galois. 

Since the number of elements $\gamma \in \mathbb{F}_{q_0}$ with $\gamma^2+\lambda=0$ is at most two, we have $\delta'(C) \ge q_0+1-2=q_0-1$. 
If there exists $\gamma \in \mathbb{F}_{q_0}$ with $\gamma^2+\lambda=0$, then $\gamma=\pm \alpha$ and hence, $\alpha \in \mathbb F_{q_0}$.
Therefore, $\delta'(C) \ge q_0+1$ if $\alpha \not\in \mathbb F_{q_0}$.  
\end{proof}

Next we consider the latter assertion of (e). 
By using the Hessian matrix (see the proof of Lemma \ref{flex}), we have the following. 

\begin{lemma} \label{flex2} 
The curve $C$ has no flex in the affine plane $Z \ne 0$. 
\end{lemma}

Now, assume that $\fq=\mathbb {F}_{q_0}$. 
Then, we have the equation 
$$ (x^q+\alpha_0x)^2+\lambda (y^q+\alpha_0 y)^2+\mu=0. $$
Making the variable change $Z \mapsto (1/\sqrt[q-1]{-\alpha_0}) Z$, we may assume that $\alpha_0=-1$. 
In this case, we prove that all outer Galois points are $\fq$-rational points on the line $Z=0$. 

We determine special multiple tangent lines. 

\begin{lemma} \label{tangent}
Let $R \in C \setminus {\rm Sing}(C)$. 
If $T_{R}C$ has distinct $q$ contact points in $C$, then $T_{R}C$ intersects the line $Z=0$ at an $\fq$-rational point.  
\end{lemma}

\begin{proof} 
The Gauss map $\gamma: C \dashrightarrow (\mathbb P^2)^* \cong \mathbb P^2$, which sends a smooth point $R$ to the tangent line $T_RC$ at $R$, is given by 
\begin{eqnarray*} 
& &  (\partial F/\partial X:\partial F/\partial Y:\partial F/\partial Z) \\ 
&=&(-2(x^q-x):-2\lambda(y^q-y):2x(x^q-x)+2\lambda y(y^q-y)) \\
&=& (x^q-x:\lambda (y^q-y):-x(x^q-x)-\lambda y(y^q-y))
\end{eqnarray*} 
Let $R_i=(x_i:y_i:1)$ for $i=1, \ldots, q$, let $R_i \ne R_j$ and let $T_{R_i}C=T_{R_j}C$ for any $i, j$. 
Assume that $x_1^q-x_1 \ne 0$. 
Then, $x_i^q-x_i \ne 0$. 
Since $T_{R_1}C=T_{R_i}C$, $(y_1^q-y_1)/(x_1^q-x_1)=(y_i^q-y_i)/(x_i^q-x_i)$ and  
$$ \frac{\mu}{(x_1^q-x_1)^2}=-1-\lambda \left(\frac{y_1^q-y_1}{x_1^q-x_1}\right)^2=\frac{\mu}{(x_i^q-x_i)^2}$$
by the defining equation, 
we have $x_1^q-x_1=\pm(x_i^q-x_i)$ and $y_1^q-y_1=\pm(y_i^q-y_i)$. 

Assume that there exists $i$ such that $x_1^q-x_1=x_i^q-x_i$. 
We may assume $i=2$. 
Then, $y_1^q-y_1=y_2^q-y_2$ and there exist $\beta_x, \beta_y \in \fq$ such that $x_2=x_1+\beta_x$ and $y_2=y_1+\beta_y$. 
If $\beta_y=0$, then $\beta_x=0$, by the condition $x_1+\lambda y_1(y_1^q-y_1)/(x_1^q-x_1)=x_2+\lambda y_2(y_2^q-y_2)/(x_2^q-x_2)$. 
Therefore, $\beta_y \ne 0$. 
By using the condition $x_1+\lambda y_1(y_1^q-y_1)/(x_1^q-x_1)=x_2+\lambda y_2(y_2^q-y_2)/(x_2^q-x_2)$, we have $\lambda (y_1^q-y_1)/(x_1^q-x_1)=-\beta_x/\beta_y$. 
Then, $T_{R_1}C$ is defined by 
$$ X-\frac{\beta_x}{\beta_y}Y+\left(-x_1+y_1\frac{\beta_x}{\beta_y}\right)Z=0,  $$
and meets the line $Z=0$ at the point $(\beta_x:\beta_y:0)$, which is $\fq$-rational.

Assume that $x_1^q-x_1=-(x_i^q-x_i)$ for any $i$. 
Then, there exists $\alpha_i \in \fq$ such that $x_i=-x_1+\alpha_i$. 
Then, $x_3-x_2=(-x_1+\alpha_3)-(-x_1+\alpha_2)=\alpha_3-\alpha_2 \in \fq$. We can reduce to the above case.

If $y_1^q+y_1 \ne 0$, then we have the same assertion, similarly to the above discussion. 
\end{proof}

\begin{lemma} \label{rationality1}
Assume that $\alpha \not\in \fq$. 
Then, any outer Galois point is an $\fq$-rational point on the line $Z=0$. 
In particular, $\delta'(C) \le q+1$. 
\end{lemma} 

\begin{proof}
Let $P'=(\gamma:1:0)$. 
We consider the projection from $P'$. 
The projection $\pi_{P'}$ from $P'$ is given by $\pi_{P'}(x:y:1)=(x-\gamma y:1)$. 
Let $t:=x-\gamma y$. 
Then, we have the field extension $K(t, y)/K(t)$ given by $ \hat{f}(t, y):=((t+\gamma y)^q-(t+\gamma y))^2+\lambda (y^q-y)^2+\mu=0$. 
We have
\begin{eqnarray*} 
\hat{f}(t, y) &=& ((t^q-t)+ (\gamma^q y^q-\gamma y))^2+\lambda (y^q-y)^2+\mu \\
&=&(\gamma^{2q}+\lambda)y^{2q}-2(\gamma^{q+1}+\lambda)y^{q+1}+(\mbox{ lower terms }). 
\end{eqnarray*} 
Note that $\gamma^{2q}+\lambda=0$ if and only if $\gamma^2+\lambda=0$, since $\lambda \in \fq$. 
If $\gamma^2+\lambda \ne 0$, the degree of $\hat{f}(t, y)$ is $2q$. 
If $\gamma^2+\lambda=0$, then $\gamma^2=\alpha^2$ and the degree is $q+1$, since $\gamma^{q+1}+\lambda=\alpha^{q+1}-\alpha^2=\alpha(\alpha^q-\alpha) \ne 0$. 
We also have   
$$ 
\frac{dt}{dy}=\frac{\gamma((t^q-t)+(\gamma^q y^q-\gamma y))+\lambda(y^q-y)}{(t^q-t)+(\gamma^q y^q-\gamma y)}. 
$$ 
We consider the condition $\hat{f}(t,y)=dt/dy=0$. 
Taking $t^q-t+(\gamma^q y^q-\gamma y)=-(\lambda/\gamma)(y^q-y)$, we have $(\lambda^2/\gamma^2+\lambda)(y^q-y)^2+\mu=0$. 
If $\lambda \ne -\gamma^2$, we have a solution $(t_0, y_0)$ satisfying $\hat{f}(t_0, y_0)=(dt/dy)(t_0, y_0)=0$. 
This implies that $\pi_{P'}$ is ramified at some point $R$ in the affine plane $Z \ne 0$. 
If $\lambda = -\gamma^2$ (this implies that $P'$ is singular), $\hat{f}(t,y)$ is of degree $q+1$ and we does not have a solution $(t_0, y_0)$ satisfying $\hat{f}(t_0, y_0)=(dt/dy)(t_0, y_0)=0$. 
This implies that $\pi_{P'}$ is not ramified at any point $R$ in the affine plane $Z \ne 0$.

Let $O \in \mathbb P^2 \setminus C$ be a point in the affine plane $Z \ne 0$ and let $Q$ be a singular point. 
Then, the line $\overline{OQ}$ passes through exactly $q+1$ smooth points with multiplicity one except for $Q$, by the above discussion. 
Since the cardinality of $\pi^{-1}(Q)$ is $<q-1$ by Proposition \ref{general property}(c), the projection $\hat{\pi}_O$ is ramified at some point in $\pi^{-1}(Q)$. 
By Fact \ref{Galois covering}(a), $O$ is not Galois. 
Therefore, all Galois points lie on the line $Z=0$. 

Let $P'=(\gamma:1:0)$ be outer Galois. 
By Lemma \ref{flex2}, $I_R(C, T_RC)=2$.
By Fact \ref{Galois covering}(a), $T_RC$ contains exactly $q$ contact points.  
By Lemma \ref{tangent}, the point given by $T_RC \cap \{Z=0\}$ is $\fq$-rational. 
Therefore, $P'$ is $\fq$-rational. 
\end{proof}

\begin{lemma} \label{rationality2}
Assume that $\alpha \in \fq$. 
Then, any outer Galois point is an $\fq$-rational point on the line $Z=0$. 
In particular, $\delta'(C) \le q-1$. 
\end{lemma}

\begin{proof}
Firstly, we prove that points on the affine plane $Z \ne 0$ are not outer Galois. 
Assume by contradiction that $P'$ in the affine plane $Z \ne 0$ is outer Galois. 
Similarly to the proof of assertion (d), we may assume that $\hat{\pi}_{P'}$ is unramified at each point in $\pi^{-1}(Q)$ for any singular point $Q$. 
Note that $\hat{\pi}_{P'}$ is ramified at some point $\hat{R} \in \hat{C}$, by the Riemann-Hurwitz formula.  
Now, $R=\pi(\hat{R}) \in C \setminus {\rm Sing}(C)$. 
By Lemma \ref{flex2}, $I_R(C, T_RC)=2$.
By Fact \ref{Galois covering}(a), $T_RC$ contains exactly $q$ contact points.  
According to the Riemann-Hurwtiz formula again, considering $p_g(C)$, there exist $d=2q$ tangent lines containing $P'$ and $q$ contact points. 
However, by Lemma \ref{tangent}, such lines are at most $q+1$. 
Since $q+1<2q=d$, this is a contradiction.

Let $P' \in \{Z=0\}$ be outer Galois.  
Similarly to the proof of assertion (d), $\pi_{P'}$ is ramified at some point $R$ in the affine plane $Z \ne 0$. 
By Lemma \ref{flex2}, $I_R(C, T_RC)=2$.
By Fact \ref{Galois covering}(a), $T_RC$ contains exactly $q$ contact points.  
By Lemma \ref{tangent}, the point given by $T_RC \cap \{Z=0\}$ is $\fq$-rational. 
Therefore, $P'$ is $\fq$-rational. 

By Proposition \ref{general property}(a), ${\rm Sing}(C)=\{(\pm \alpha:1:0)\}$. 
The two singular points are $\fq$-rational. 
Therefore, $\delta'(C) \le q+1-2=q-1$. 
\end{proof}

As a consequence of Lemmas \ref{rationality1} and \ref{rationality2}, we have the latter assertion of (e).

\

\begin{center} {\bf Acknowledgements} \end{center} 
The author was partially supported by JSPS KAKENHI Grant Numbers 22740001, 25800002.

\end{document}